\documentclass{amsart}%
\usepackage{amssymb}
\usepackage{amsfonts}%
\usepackage{amsmath}%
\usepackage{graphicx}
\providecommand{\U}[1]{\protect \rule{.1in}{.1in}}
\newtheorem{theorem}{Theorem}
\theoremstyle{plain}

\newtheorem{corollary}{Corollary}

\newtheorem{definition}{Definition}

\newtheorem{lemma}{Lemma}

\newtheorem{remark}{Remark}

\numberwithin{equation}{section}
\begin{document}
\title[Adams-Spanne type estimates for the commutators]{Adams-Spanne type estimates for the commutators of fractional type sublinear
operators in generalized Morrey spaces on Heisenberg groups }
\author{F.GURBUZ}
\address{ANKARA UNIVERSITY, FACULTY OF SCIENCE, DEPARTMENT OF MATHEMATICS, TANDO\u{G}AN
06100, ANKARA, TURKEY }
\curraddr{}
\email{feritgurbuz84@hotmail.com}
\urladdr{}
\thanks{}
\thanks{}
\thanks{}
\date{}
\dedicatory{ }
\begin{abstract}
In this paper we give $BMO$ (bounded mean oscillation) {space} estimates for
commutators of fractional type sublinear operators in generalized Morrey
spaces on Heisenberg groups. The boundedness conditions are also formulated in
terms of Zygmund type integral inequalities.

\end{abstract}
\maketitle

\section{Introduction and main results}

Heisenberg groups play an important role in several branches of mathematics,
such as quantum physics, Fourier analysis, several complex variables, geometry
and topology; see \cite{Stein93} for more details. It is a remarkable fact
that the Heisenberg group, denoted by $H_{n}$, arises in two aspects. On the
one hand, it can be realized as the boundary of the unit ball in several
complex variables. On the other hand, an important aspect of the study of the
Heisenberg group is the background of physics, namely, the mathematical ideas
connected with the fundamental notions of quantum mechanics. In other words,
there is its genesis in the context of quantum mechanics which emphasizes its
symplectic role in the theory of theta functions and related parts of
analysis. Analysis on the groups is also motivated by their role as the
simplest and the most important model in the general theory of vector fields
satisfying H\"{o}rmander's condition. Due to this reason, many interesting
works have been devoted to the theory of harmonic analysis on $H_{n}$ in
\cite{Muller1, Folland-Stein, Garofalo, Muller2, Muller3, Stein93, Xiao,
Yener}.

We start with some basic knowledge about Heisenberg group in generalized
Morrey spaces and refer the reader to \cite{Folland-Stein, Guliyev, Garofalo,
Stein93} and the references therein for more details. The Heisenberg group
$H_{n}$ is a non-commutative nilpotent Lie group, with the underlying manifold
$%
\mathbb{R}
^{2n}\times%
\mathbb{R}
$ and the group structure is given by
\[
\left(  x,t\right)  \circ \left(  x^{\prime},t^{\prime}\right)  =\left(
x+x^{\prime},t+t^{\prime}+2\sum_{j=1}^{n}\left(  x_{2j}x_{2j-1}^{\prime
}-x_{2j-1}x_{2j}^{\prime}\right)  \right)  .
\]
Using the coordinates $g=\left(  x,t\right)  $ for points in $H_{n}$, the
left-invariant vector fields for this group structure are%
\[
X_{2j-1}=\frac{\partial}{\partial x_{2j-1}}+2x_{2j}\frac{\partial}{\partial
t},\qquad \text{\ }j=1,\ldots,n,
\]%
\[
X_{2j}=\frac{\partial}{\partial x_{2j}}+2x_{2j-1}\frac{\partial}{\partial
t},\qquad \text{\ }j=1,\ldots,n.
\]
These vector fields generate the Lie algebra of $H_{n}$ and the commutators of
the vector fields $\left(  X_{1},\ldots,X_{2n}\right)  $ satisfy the relation%
\[
\left[  X_{j},X_{n+j}\right]  =-4X_{2n+1},\qquad j=1,\ldots,n,
\]
with all other brackets being equal to zero.

The inverse element of $g=\left(  x,t\right)  $ is $g^{-1}=\left(
-x,-t\right)  $ and we denote the identity (neutral) element of $H_{n}$ as
$e=\left(  0,0\right)  \in%
\mathbb{R}
^{2n+1}$. The Heisenberg group is a connected, simply connected nilpotent Lie
group. One-parameter Heisenberg dilations $\delta_{r}:$ $H_{n}\rightarrow$
$H_{n}$ are given by $\delta_{r}\left(  x,t\right)  =\left(  rx,r^{2}t\right)
$ for each real number $r>0$. The Haar measure on $H_{n}$ also coincides with
the usual Lebesgue measure on $%
\mathbb{R}
^{2n+1}$. These dilations are group automorphisms and Jacobian determinant of
$\delta_{r}$ with respect to the Lebesgue measure is equal to $r^{Q}$, where
$Q=2n+2$ is the homogeneous dimension of $H_{n}$. We denote the measure of any
measurable set $\Omega \subset$ $H_{n}$ by $\left \vert \Omega \right \vert $.
Then%
\[
\left \vert \delta_{r}\left(  \Omega \right)  \right \vert =r^{Q}\left \vert
\Omega \right \vert ,\qquad d\left(  \delta_{r}x\right)  =r^{Q}dx.
\]
The homogeneous norm on $H_{n}$ is defined as follows%
\[
\left \Vert x\right \Vert _{H_{n}}=\left \Vert \left(  x_{1},\ldots
,x_{2n},x_{2n+1}\right)  \right \Vert _{H_{n}}=\left[  \left(  \sum_{j=1}%
^{2n}x_{j}^{2}\right)  ^{2}+x_{2n+1}^{2}\right]  ^{1/4},
\]
and the Heisenberg distance is given by%
\[
d\left(  g,h\right)  =d\left(  g^{-1}h,0\right)  =\left \vert g^{-1}%
h\right \vert .
\]
This distance $d$ is left-invariant in the sense that $d\left(  g,h\right)
=\left \vert g^{-1}h\right \vert $ remains unchanged when $g$ and $h$ are both
left-translated by some fixed vector on $H_{n}$. Moreover, $d$ satisfies the
triangular inequality (see \cite{Koranyi}, page 320)%
\[
d\left(  g,h\right)  \leq d\left(  g,x\right)  +d\left(  x,h\right)  ,\qquad
g,x,h\in H_{n}.
\]
Using this norm, we define the Heisenberg ball%
\[
B\left(  g,r\right)  =\left \{  h\in H_{n}:\left \vert g^{-1}h\right \vert
<r\right \}
\]
with center $g=\left(  x,t\right)  $ and radius $r$ and denote by
$B^{C}\left(  g,r\right)  =H_{n}\setminus B\left(  g,r\right)  $ its
complement, and we denote by $B_{r}=B\left(  e,r\right)  =\left \{  h\in
H_{n}:\left \vert h\right \vert <r\right \}  $ the open ball centered at $e$, the
identity (neutral) element of $H_{n}$, with radius $r$. The volume of the ball
$B\left(  g,r\right)  $ is $c_{Q}r^{Q}$, where $c_{n}$ is the volume of the
unit ball $B_{1}$:%
\[
c_{Q}=\left \vert B\left(  e,1\right)  \right \vert =\frac{2\pi^{n+\frac{1}{2}%
}\Gamma \left(  \frac{n}{2}\right)  }{\left(  n+1\right)  \Gamma \left(
n\right)  \Gamma \left(  \frac{n+1}{2}\right)  }.
\]
For more details about Heisenberg group, one can refer to \cite{Folland-Stein}.

In the study of local properties of solutions to of second order elliptic
partial differential equations(PDEs), together with weighted Lebesgue spaces,
Morrey spaces $L_{p,\lambda}\left(  H_{n}\right)  $ play an important role,
see \cite{Giaquinta, Kufner}. They were introduced by Morrey in 1938
\cite{Morrey}. For the properties and applications of classical Morrey spaces,
see \cite{ChFraL1, ChFraL2, Gurbuz} and the references therein. We recall its
definition on a Heisenberg group as%

\[
L_{p,\lambda}\left(  H_{n}\right)  =\left \{  f:\left \Vert f\right \Vert
_{L_{p,\lambda}\left(  H_{n}\right)  }=\sup \limits_{g\in H_{n},r>0}%
\,r^{-\frac{\lambda}{p}}\, \Vert f\Vert_{L_{p}(B(g,r))}<\infty \right \}  ,
\]
where $f\in L_{p}^{loc}(H_{n})$, $0\leq \lambda \leq Q$ and $1\leq p<\infty$.

Note that $L_{p,0}=L_{p}(H_{n})$ and $L_{p,Q}=L_{\infty}(H_{n})$. If
$\lambda<0$ or $\lambda>Q$, then $L_{p,\lambda}={\Theta}$, where $\Theta$ is
the set of all functions equivalent to $0$ on $H_{n}$ . It is known that
$L_{p,\lambda}(H_{n})$ is an expansion of $L_{p}(H_{n})$ in the sense that
$L_{p,0}=L_{p}(H_{n})$.

We also denote by $WL_{p,\lambda}\equiv WL_{p,\lambda}(H_{n})$ the weak Morrey
space of all functions $f\in WL_{p}^{loc}(H_{n})$ for which
\[
\left \Vert f\right \Vert _{WL_{p,\lambda}}\equiv \left \Vert f\right \Vert
_{WL_{p,\lambda}(H_{n})}=\sup_{g\in H_{n},r>0}r^{-\frac{\lambda}{p}}\Vert
f\Vert_{WL_{p}(B(g,r))}<\infty,
\]
where $WL_{p}(B(g,r))$ denotes the weak $L_{p}$-space of measurable functions
$f$ for which
\[%
\begin{split}
\Vert f\Vert_{WL_{p}(B(g,r))} &  \equiv \Vert f\chi_{_{B(g,r)}}\Vert
_{WL_{p}(H_{n})}\\
&  =\sup_{\tau>0}\tau \left \vert \left \{  h\in B(g,r):\,|f(h)|>\tau \right \}
\right \vert ^{1/{p}}\\
&  =\sup_{0<\tau \leq|B(g,r)|}\tau^{1/{p}}\left(  f\chi_{_{B(g,r)}}\right)
^{\ast}(\tau)<\infty.
\end{split}
\]
Here $g^{\ast}$ denotes the non-increasing rearrangement of a function $g$.

Note that%
\[
WL_{p}\left(  H_{n}\right)  =WL_{p,0}\left(  H_{n}\right)  \text{,
\  \ }L_{p,\lambda}\left(  H_{n}\right)  \subset WL_{p,\lambda}\left(
H_{n}\right)  \text{ \  \ and \  \ }\left \Vert f\right \Vert _{WL_{p,\lambda
}(H_{n})}\leq \left \Vert f\right \Vert _{L_{p,\lambda}\left(  H_{n}\right)  }.
\]

Let $|B(g,r)|$ be the Haar measure of the ball $B(g,r)$. Let $f$ be a given
integrable function on a ball $B\left(  g,r\right)  \subset G$. The fractional
maximal function $M_{\alpha}f$, $0\leq \alpha<Q$, of $f$ is defined by the
formula
\[
M_{\alpha}f(g)=\sup_{r>0}|B(g,r)|^{-1+\frac{\alpha}{Q}}\int \limits_{B(g,r)}%
|f(h)|dh.
\]

In the case of $\alpha=0$, the fractional maximal function $M_{\alpha}f$
coincides with the Hardy-Littlewood maximal function $Mf\equiv M_{0}f$ (see
\cite{Folland-Stein, Stein93}) and is closely related to the fractional
integral%
\[
\overline{T}_{\alpha}f\left(  g\right)  =\int \limits_{H_{n}}\frac{f\left(
h\right)  }{\left \vert g^{-1}h\right \vert ^{Q-\alpha}}dh\qquad0<\alpha<Q.
\]

The operators $M_{\alpha}$ and $\overline{T}_{\alpha}$ play an important role
in real and harmonic analysis (see \cite{Folland, Folland-Stein, Stein93,
Xiao}).

The classical Riesz potential $I_{\alpha}$ is defined on ${\mathbb{R}^{n}}$ by
the formula%
\[
I_{\alpha}f=\left(  -\Delta \right)  ^{-\frac{\alpha}{2}}f,\qquad0<\alpha<n,
\]
where $\Delta$ is the Laplacian operator. It is known that%
\[
I_{\alpha}f\left(  x\right)  =\frac{1}{\gamma \left(  \alpha \right)  }%
\int \limits_{{\mathbb{R}^{n}}}\frac{f\left(  y\right)  }{\left \vert
x-y\right \vert ^{n-\alpha}}dy\equiv \overline{T}_{\alpha}f\left(  x\right)  ,
\]
where $\gamma \left(  \alpha \right)  =\pi^{\frac{n}{2}}2^{\alpha}\frac
{\Gamma \left(  \frac{\alpha}{2}\right)  }{\Gamma \left(  \frac{n-\alpha}%
{2}\right)  }$. The Riesz potential on the Heisenberg group is defined in
terms of the sub-Laplacian $%
\mathcal{L}%
$\ $=\Delta_{H_{n}}$.

\begin{definition}
For $0<\alpha<Q$ the Riesz potential $I_{\alpha}$ is defined by on the
Schwartz space $S\left(  H_{n}\right)  $ by the formula%
\[
I_{\alpha}f\left(  g\right)  =%
\mathcal{L}%
^{-\frac{\alpha}{2}}f\left(  g\right)  \equiv%
{\displaystyle \int \limits_{0}^{\infty}}
e^{-r%
\mathcal{L}%
}f\left(  g\right)  r^{\frac{\alpha}{2}-1}dr,
\]
where%
\[
e^{-r%
\mathcal{L}%
}f\left(  g\right)  =\frac{1}{\Gamma \left(  \frac{\alpha}{2}\right)  }%
{\displaystyle \int \limits_{H_{n}}}
K_{r}\left(  h,g\right)  f\left(  h\right)  d\left(  h\right)
\]
is the semigroups of the operator $%
\mathcal{L}%
$.
\end{definition}

In \cite{Xiao}, relations between the Riesz potential and the heat kernel on
the Heisenberg group are studied. The following assertion $\left[
\text{\cite{Xiao}, Theorem 4.2}\right]  $ yields an expression for $I_{\alpha
}$, which allows us to discuss the boundedness of the Riesz potential.

\begin{theorem}
Let $q_{s}\left(  g\right)  $ be the heat kernel on $H_{n}$. If $0\leq
\alpha<Q$, then for $f\in S\left(  H_{n}\right)  $%
\[
I_{\alpha}f\left(  g\right)  =\frac{1}{\Gamma \left(  \frac{\alpha}{2}\right)
}%
{\displaystyle \int \limits_{0}^{\infty}}
s^{\frac{\alpha}{2}-1}q_{s}\left(  \cdot \right)  ds\ast f\left(  g\right)  .
\]

\end{theorem}

The Riesz potential $I_{\alpha}$ satisfies the estimate $\left[
\text{\cite{Xiao}, Theorem 4.4}\right]  $%
\[
\left \vert I_{\alpha}f\left(  g\right)  \right \vert \lesssim \overline
{T}_{\alpha}f\left(  g\right)  ,
\]
which provides a suitable estimate for the Riesz potential on the Heisenberg
group. It is well known that, see \cite{Folland-Stein, Stein93} for example,
$\overline{T}_{\alpha}$ is bounded from $L_{p}\left(  H_{n}\right)  $ to
$L_{q}\left(  H_{n}\right)  $ for all $p>1$ and $\frac{1}{p}-$ $\frac{1}{q}=$
$\frac{\alpha}{Q}>0$, and $\overline{T}_{\alpha}$ is also of weak type
$\left(  1,\frac{Q}{Q-\alpha}\right)  $(i.e. Hardy-Littlewood Sobolev inequality).

Spanne (published by Peetre \cite{Peetre}) and Adams \cite{Adams} have studied
boundedness of the fractional integral operator $\overline{T}_{\alpha}$ on
$L_{p,\lambda}\left(  {\mathbb{R}^{n}}\right)  $. This result has been
reproved by Chiarenza and Frasca \cite{ChFra}, and also studied in
\cite{Gurbuz0}.

After studying Morrey spaces in detail, researchers have passed to generalized
Morrey spaces. Recall that the concept of the generalized Morrey space
$M_{p,\varphi}\equiv M_{p,\varphi}(H_{n})$ on Heisenberg group has been
introduced in \cite{Guliyev}.

\begin{definition}
\cite{Guliyev} Let $\varphi(g,r)$ be a positive measurable function on
$H_{n}\times(0,\infty)$ and $1\leq p<\infty$. We denote by $M_{p,\varphi
}\equiv M_{p,\varphi}(H_{n})$ the generalized Morrey space, the space of all
functions $f\in L_{p}^{loc}(H_{n})$ with finite quasinorm
\[
\Vert f\Vert_{M_{p,\varphi}}=\sup \limits_{g\in H_{n},r>0}\varphi
(g,r)^{-1}\,|B(g,r)|^{-\frac{1}{p}}\, \Vert f\Vert_{L_{p}(B(g,r))}.
\]
Also by $WM_{p,\varphi}\equiv WM_{p,\varphi}(H_{n})$ we denote the weak
generalized Morrey space of all functions $f\in WL_{p}^{loc}(H_{n})$ for
which
\[
\Vert f\Vert_{WM_{p,\varphi}}=\sup \limits_{g\in H_{n},r>0}\varphi
(g,r)^{-1}\,|B(g,r)|^{-\frac{1}{p}}\, \Vert f\Vert_{WL_{p}(B(g,r))}<\infty.
\]

\end{definition}

According to this definition, we recover the Morrey space $L_{p,\lambda}$ and
weak Morrey space $WL_{p,\lambda}$ under the choice $\varphi(g,r)=r^{\frac
{\lambda-Q}{p}}$:
\[
L_{p,\lambda}=M_{p,\varphi}\mid_{\varphi(g,r)=r^{\frac{\lambda-Q}{p}}%
},~~~~~~~~WL_{p,\lambda}=WM_{p,\varphi}\mid_{\varphi(g,r)=r^{\frac{\lambda
-Q}{p}}}.
\]
In\  \cite{Guliyev}, Guliyev et al. prove the Spanne type boundedness of Riesz
potentials $I_{\alpha}$, $\alpha \in \left(  0,Q\right)  $ from one generalized
Morrey space $M_{p,\varphi_{1}}\left(  H_{n}\right)  $ to another
$M_{q,\varphi_{2}}\left(  H_{n}\right)  $, where $1<p<q<\infty$, $\frac{1}%
{p}-\frac{1}{q}=\frac{\alpha}{Q}$, $Q$ is the homogeneous dimension of $H_{n}$
and from the space $M_{1,\varphi_{1}}\left(  H_{n}\right)  $ to the weak space
$WM_{1,\varphi_{2}}\left(  H_{n}\right)  $, where $1<q<\infty$, $1-\frac{1}%
{q}=\frac{\alpha}{Q}$. They also prove the Adams type boundedness of the Riesz
potentials $I_{\alpha}$, $\alpha \in \left(  0,Q\right)  $ from $M_{p,\varphi
^{\frac{1}{p}}}\left(  H_{n}\right)  $ to another $M_{q,\varphi^{\frac{1}{q}}%
}\left(  H_{n}\right)  $ for $1<p<q<\infty$ and from the space $M_{1,\varphi
}\left(  H_{n}\right)  $ to the weak space $WM_{1,\varphi^{\frac{1}{q}}%
}\left(  H_{n}\right)  $ for $1<q<\infty$.

For a locally integrable function $b$ on $H_{n}$, suppose that the commutator
operator $T_{b,\alpha}$, $\alpha \in \left(  0,Q\right)  $ represents a linear
or a sublinear operator, which satisfies that for any $f\in L_{1}(H_{n})$ with
compact support and $x\notin suppf$
\begin{equation}
|T_{b,\alpha}f(g)|\leq c_{0}%
{\displaystyle \int \limits_{H_{n}}}
|b(g)-b(h)|\, \frac{|f(h)|}{|g^{-1}h|^{Q-\alpha}}dh,\label{e2}%
\end{equation}
where $c_{0}$ is independent of $f$ and $g$.

The condition (\ref{e2}) is satisfied by many interesting operators in
harmonic analysis, such as fractional maximal operator, fractional
Marcinkiewicz operator, fractional integral operator and so on (see
\cite{LLY}, \cite{SW} for details).

Let $T$ be a linear operator. For a locally integrable function $b$ on $H_{n}
$, we define the commutator $[b,T]$ by
\[
\lbrack b,T]f(x)=b(x)\,Tf(x)-T(bf)(x)
\]
for any suitable function $f$.

Let $b$ be a locally integrable function on $H_{n}$, then for $0<\alpha<Q$, we
define the linear commutator generated by fractional integral operator and $b$
and the sublinear commutator of the fractional maximal operator as follows,
respectively (see also \cite{LLY}).
\[
\lbrack b,\overline{T}_{\alpha}]f(g)\equiv b(g)\overline{T}_{\alpha
}f(g)-\overline{T}_{\alpha}(bf)(g)=\int \limits_{H_{n}}[b(g)-b(h)]\frac
{f(h)}{|g^{-1}h|^{Q-\alpha}}dh,
\]%
\[
M_{b,\alpha}\left(  f\right)  (g)=\sup_{r>0}|B(g,r)|^{-1+\frac{\alpha}{Q}}%
\int \limits_{B(g,r)}\left \vert b\left(  g\right)  -b\left(  h\right)
\right \vert |f(h)|dh.
\]

Now, we will examine some properties related to the space of functions of
Bounded Mean Oscillation, $BMO$, introduced by John and Nirenberg
\cite{John-Nirenberg} in 1961. This space has become extremely important in
various areas of analysis including harmonic analysis, PDEs and function
theory. $BMO$-spaces are also of interest since, in the scale of Lebesgue
spaces, they may be considered and appropriate substitute for $L_{\infty}$.
Appropriate in the sense that are spaces preserved by a wide class of
important operators such as the Hardy-Littlewood maximal function, the Hilbert
transform and which can be used as an end point in interpolating $L_{p}$ spaces.

Let us recall the defination of the space of $BMO(H_{n})$ (see, for example,
\cite{Folland-Stein, LLY, Stromberg}).

\begin{definition}
Suppose that $b\in L_{1}^{loc}(H_{n})$, let
\begin{equation}
\Vert b\Vert_{\ast}=\sup_{g\in H_{n},r>0}\frac{1}{|B(g,r)|}%
{\displaystyle \int \limits_{B(g,r)}}
|b(h)-b_{B(g,r)}|dh<\infty,\label{1*}%
\end{equation}
where
\[
b_{B(g,r)}=\frac{1}{|B(g,r)|}%
{\displaystyle \int \limits_{B(g,r)}}
b(h)dh.
\]
Define
\[
BMO(H_{n})=\{b\in L_{1}^{loc}(H_{n})~:~\Vert b\Vert_{\ast}<\infty \}.
\]

Endowed with the norm given in (\ref{1*}), $BMO(H_{n})$ becomes Banach space
provided we identify functions which differ a.e. by constant; clearly, $\Vert
b\Vert_{\ast}=0$ for $b\left(  h\right)  =c$ a.e. in $H_{n}$.
\end{definition}

\begin{remark}
Note that $L_{\infty}(H_{n})$ is contained in $BMO(H_{n})$ and we have
\[
\Vert b\Vert_{\ast}\leq2\Vert b\Vert_{\infty}.
\]
Moreover, $BMO$ contains unbounded functions, in fact the function
log$\left \vert h\right \vert $ on $H_{n}$, is in $BMO$ but it is not bounded,
so $L_{\infty}(H_{n})\subset BMO(H_{n})$.
\end{remark}

\begin{remark}
$(1)~~$ The John-Nirenberg inequality \cite{John-Nirenberg}: there are
constants $C_{1}$, $C_{2}>0$, such that for all $b\in BMO(H_{n})$ and
$\beta>0$
\[
\left \vert \left \{  g\in B\,:\,|b(g)-b_{B}|>\beta \right \}  \right \vert \leq
C_{1}|B|e^{-C_{2}\beta/\Vert b\Vert_{\ast}},~~~\forall B\subset H_{n}.
\]

$(2)~~$ The John-Nirenberg inequality implies that
\begin{equation}
\Vert b\Vert_{\ast}\thickapprox \sup_{g\in H_{n},r>0}\left(  \frac{1}{|B(g,r)|}%
{\displaystyle \int \limits_{B(g,r)}}
|b(h)-b_{B(g,r)}|^{p}dh\right)  ^{\frac{1}{p}}\label{5.1}%
\end{equation}
for $1<p<\infty$.

$(3)~~$ Let $b\in BMO(H_{n})$. Then there is a constant $C>0$ such that
\begin{equation}
\left \vert b_{B(g,r)}-b_{B(g,\tau)}\right \vert \leq C\Vert b\Vert_{\ast}%
\ln \frac{\tau}{r}~\text{for}~0<2r<\tau,\label{5.2}%
\end{equation}
where $C$ is independent of $b$, $g$, $r$ and $\tau$.
\end{remark}

Inspired by \cite{Guliyev}, in this paper, provided that $b\in BMO\left(
H_{n}\right)  $ and $T_{b,\alpha}$, $\alpha \in \left(  0,Q\right)  $ satisfying
condition (\ref{e2}) is a sublinear operator, we find the sufficient
conditions on the pair $(\varphi_{1},\varphi_{2})$ which ensures the Spanne
type boundedness of the commutator operators $T_{b,\alpha}$ from
$M_{p,\varphi_{1}}\left(  H_{n}\right)  $ to $M_{q,\varphi_{2}}\left(
H_{n}\right)  $, where $1<p<q<\infty$, $0<\alpha<\frac{Q}{p}$, $\frac{1}%
{q}=\frac{1}{p}-\frac{\alpha}{Q}$. We also find the sufficient conditions on
$\varphi$ which ensures the Adams type boundedness of the commutator operators
$T_{b,\alpha}$ from $M_{p,\varphi^{\frac{1}{p}}}\left(  H_{n}\right)  $ to
another $M_{q,\varphi^{\frac{1}{q}}}\left(  H_{n}\right)  $ for $1<p<q<\infty
$. In all the cases the conditions for the boundedness of $T_{b,\alpha}$ are
given in terms of Zygmund-type (supremal-type) integral inequalities on
$\left(  \varphi_{1},\varphi_{2}\right)  $ and $\varphi$ which do not assume
any assumption on monotonicity of $\varphi_{1},\varphi_{2}$ and $\varphi$ in
$r$. Our main results can be formulated as follows.

\begin{theorem}
\label{teo15}(Spanne type result) Let $1<p<\infty$, $0<\alpha<\frac{Q}{p}$,
$\frac{1}{q}=\frac{1}{p}-\frac{\alpha}{Q}$ and $b\in BMO\left(  H_{n}\right)
$. Let $T_{b,\alpha}$ be a sublinear operator satisfying condition (\ref{e2})
and bounded from $L_{p}(H_{n})$ to $L_{q}(H_{n})$. Let also, the pair
$(\varphi_{1},\varphi_{2})$ satisfies the condition%
\begin{equation}
\int \limits_{r}^{\infty}\left(  1+\ln \frac{\tau}{r}\right)  \frac
{\operatorname*{essinf}\limits_{\tau<s<\infty}\varphi_{1}\left(  g,s\right)
s^{\frac{Q}{p}}}{\tau^{\frac{Q}{q}+1}}dt\leq C\varphi_{2}\left(  g,r\right)
,\label{47}%
\end{equation}

Then, the operator $T_{b,\alpha}$ is bounded from $M_{p,\varphi_{1}}\left(
H_{n}\right)  $ to $M_{q,\varphi_{2}}\left(  H_{n}\right)  $. Moreover%
\[
\left \Vert T_{b,\alpha}f\right \Vert _{M_{q,\varphi_{2}}}\lesssim \left \Vert
b\right \Vert _{\ast}\left \Vert f\right \Vert _{M_{p,\varphi_{1}}}.
\]

\end{theorem}

From Theorem \ref{teo15} we get the following new result.

\begin{corollary}
Let $1<p<\infty$, $0<\alpha<\frac{Q}{p}$, $\frac{1}{q}=\frac{1}{p}%
-\frac{\alpha}{Q}$, $b\in BMO\left(  H_{n}\right)  $ and the pair
$(\varphi_{1},\varphi_{2})$ satisfies condition (\ref{47}). Then, the
operators $M_{b,\alpha}$ and $[b,\overline{T}_{\alpha}]$ are bounded from
$M_{p,\varphi_{1}}\left(  H_{n}\right)  $ to $M_{q,\varphi_{2}}\left(
H_{n}\right)  $.
\end{corollary}

\begin{theorem}
\label{teo100}(Adams type result) Let $1<p<q<\infty$, $0<\alpha<\frac{Q}{p} $,
$b\in BMO\left(  H_{n}\right)  $ and let $\varphi \left(  g,\tau \right)  $
satisfies the conditions%
\begin{equation}
\sup_{r<\tau<\infty}\tau^{-\frac{Q}{p}}\left(  1+\ln \frac{\tau}{r}\right)
^{p}\operatorname*{essinf}\limits_{\tau<s<\infty}\varphi \left(  g,s\right)
s^{\frac{Q}{p}}\leq C\varphi \left(  g,r\right)  ,\label{67}%
\end{equation}
and%
\begin{equation}
\int \limits_{r}^{\infty}\left(  1+\ln \frac{\tau}{r}\right)  \tau^{\alpha
}\varphi \left(  g,\tau \right)  ^{\frac{1}{p}}\frac{d\tau}{\tau}\leq
Cr^{-\frac{\alpha p}{q-p}},\label{74}%
\end{equation}
where $C$ does not depend on $g\in H_{n}$ and $r>0$. Let also $T_{b,\alpha}$
be a sublinear operator satisfying condition (\ref{e2}) and the condition%
\begin{equation}
\left \vert T_{b,\alpha}\left(  f\chi_{B\left(  g,r\right)  }\right)  \left(
g\right)  \right \vert \lesssim r^{\alpha}M_{b}f\left(  g\right) \label{100}%
\end{equation}
holds for any ball $B\left(  g,r\right)  $.

Then the operator $T_{b,\alpha}$ is bounded from $M_{p,\varphi^{\frac{1}{p}}%
}\left(  H_{n}\right)  $ to $M_{q,\varphi^{\frac{1}{q}}}\left(  H_{n}\right)
$. Moreover, we have%
\[
\left \Vert T_{b,\alpha}f\right \Vert _{M_{q,\varphi^{\frac{1}{q}}}}%
\lesssim \Vert b\Vert_{\ast}\left \Vert f\right \Vert _{M_{p,\varphi^{\frac{1}%
{p}}}}.
\]

\end{theorem}

From Theorem \ref{teo100}, we get the following new result.

\begin{corollary}
Let $1<p<\infty$, $0<\alpha<\frac{Q}{p}$, $p<q$, $b\in BMO\left(
H_{n}\right)  $ and let also $\varphi \left(  x,\tau \right)  $ satisfies
conditions (\ref{67}) and (\ref{74}). Then the operators $M_{b,\alpha}$ and
$[b,\overline{T}_{\alpha}]$ are bounded from $M_{p,\varphi^{\frac{1}{p}}%
}\left(  H_{n}\right)  $ to $M_{q,\varphi^{\frac{1}{q}}}\left(  H_{n}\right)
$.
\end{corollary}

At last, throughout the paper we use the letter $C$ for a positive constant,
independent of appropriate parameters and not necessarily the same at each
occurrence.By $A\lesssim B$ we mean that $A\leq CB$ with some positive
constant $C$ independent of appropriate quantities. If $A\lesssim B$ and
$B\lesssim A$, we write $A\approx B$ and say that $A$ and $B$ are equivalent.

\section{Some Lemmas}

To prove the main results (Theorems \ref{teo15} and \ref{teo100}), we need the
following lemmas. Firstly, for the proof of Spanne type results, we need
following Lemma \ref{Lemma 5}.

\begin{lemma}
\label{Lemma 5}(Our main lemma) Let $1<p<\infty$, $0<\alpha<\frac{Q}{p}$,
$\frac{1}{q}=\frac{1}{p}-\frac{\alpha}{Q}$, $b\in BMO\left(  H_{n}\right)  $,
and $T_{b,\alpha}$ is a sublinear operator satisfying condition (\ref{e2}) and
bounded from $L_{p}(H_{n})$ to $L_{q}(H_{n})$. Then, the inequality
\begin{equation}
\Vert T_{b,\alpha}f\Vert_{L_{q}(B(g,r))}\lesssim \Vert b\Vert_{\ast}%
\,r^{\frac{Q}{q}}\int \limits_{2r}^{\infty}\left(  1+\ln \frac{\tau}{r}\right)
\tau^{-\frac{Q}{q}-1}\Vert f\Vert_{L_{p}(B(g,\tau))}d\tau \label{40}%
\end{equation}
holds for any ball $B(g,r)$ and for all $f\in L_{p}^{loc}(H_{n})$.
\end{lemma}

\begin{proof}
Let $1<p<\infty$, $0<\alpha<\frac{n}{p}$ and $\frac{1}{q}=\frac{1}{p}%
-\frac{\alpha}{n}$. For an arbitrary ball $B=B\left(  g,r\right)  $ we set
$f=f_{1}+f_{2}$, where \ $f_{1}=f\chi_{2B}$, \ $f_{2}=f\chi_{\left(
2B\right)  ^{C}}$ and $2B=B\left(  g,2r\right)  $. Then we have
\[
\left \Vert T_{b,\alpha}f\right \Vert _{L_{q}\left(  B\right)  }\leq \left \Vert
T_{b,\alpha}f_{1}\right \Vert _{L_{q}\left(  B\right)  }+\left \Vert
T_{b,\alpha}f_{2}\right \Vert _{L_{q}\left(  B\right)  }.
\]
From the boundedness of $T_{b,\alpha}$ from $L_{p}(H_{n})$ to $L_{q}(H_{n})$
(see, for example, \cite{Folland-Stein, Stromberg}) it follows that:%
\begin{align*}
\left \Vert T_{b,\alpha}f_{1}\right \Vert _{L_{q}\left(  B\right)  }  &
\leq \left \Vert T_{b,\alpha}f_{1}\right \Vert _{L_{q}\left(  H_{n}\right)  }\\
& \lesssim \left \Vert b\right \Vert _{\ast}\left \Vert f_{1}\right \Vert
_{L_{p}\left(  H_{n}\right)  }=\left \Vert b\right \Vert _{\ast}\left \Vert
f\right \Vert _{L_{p}\left(  2B\right)  }.
\end{align*}
It is known that $g\in B$, $h\in \left(  2B\right)  ^{C}$, which implies
$\frac{1}{2}\left \vert h^{-1}w\right \vert \leq \left \vert g^{-1}h\right \vert
\leq \frac{3}{2}\left \vert h^{-1}w\right \vert $. Then for $g\in B$, we have%
\[
\left \vert T_{b,\alpha}f_{2}\left(  g\right)  \right \vert \lesssim
\int \limits_{\left(  2B\right)  ^{C}}\left \vert b\left(  h\right)  -b\left(
g\right)  \right \vert \frac{\left \vert f\left(  h\right)  \right \vert
}{\left \vert g^{-1}h\right \vert ^{Q-\alpha}}dh.
\]
Hence we get%
\begin{align*}
\left \Vert T_{b,\alpha}f_{2}\right \Vert _{L_{q}\left(  B\right)  }  &
\lesssim \left(  \int \limits_{B}\left(  \int \limits_{\left(  2B\right)  ^{C}%
}\left \vert b\left(  h\right)  -b\left(  g\right)  \right \vert \frac
{\left \vert f\left(  y\right)  \right \vert }{\left \vert g^{-1}h\right \vert
^{Q-\alpha}}dh\right)  ^{q}dg\right)  ^{\frac{1}{q}}\\
& \lesssim \left(  \int \limits_{B}\left(  \int \limits_{\left(  2B\right)  ^{C}%
}\left \vert b\left(  h\right)  -b\left(  g\right)  \right \vert \frac
{\left \vert f\left(  h\right)  \right \vert }{\left \vert g^{-1}h\right \vert
^{Q-\alpha}}dh\right)  ^{q}dg\right)  ^{\frac{1}{q}}\\
& +\left(  \int \limits_{B}\left(  \int \limits_{\left(  2B\right)  ^{C}%
}\left \vert b\left(  h\right)  -b\left(  g\right)  \right \vert \frac
{\left \vert f\left(  y\right)  \right \vert }{\left \vert g^{-1}h\right \vert
^{Q-\alpha}}dh\right)  ^{q}dg\right)  ^{\frac{1}{q}}\\
& =J_{1}+J_{2}.
\end{align*}
We have the following estimation of $J_{1}$. When $\frac{1}{\mu}+\frac{1}%
{p}=1$, by the Fubini's theorem%
\begin{align*}
J_{1}  & \approx r^{\frac{Q}{q}}\int \limits_{\left(  2B\right)  ^{C}%
}\left \vert b\left(  h\right)  -b_{B}\right \vert \frac{\left \vert f\left(
h\right)  \right \vert }{\left \vert g^{-1}h\right \vert ^{Q-\alpha}}dh\\
& \approx r^{\frac{Q}{q}}\int \limits_{\left(  2B\right)  ^{C}}\left \vert
b\left(  h\right)  -b_{B}\right \vert \left \vert f\left(  h\right)  \right \vert
\int \limits_{\left \vert g^{-1}h\right \vert }^{\infty}\frac{d\tau}%
{\tau^{Q+1-\alpha}}dh\\
& \approx r^{\frac{Q}{q}}\int \limits_{2r}^{\infty}\int \limits_{2r\leq
\left \vert g^{-1}h\right \vert \leq \tau}\left \vert b\left(  h\right)
-b_{B}\right \vert \left \vert f\left(  h\right)  \right \vert dh\frac{d\tau
}{\tau^{Q+1-\alpha}}\\
& \lesssim r^{\frac{Q}{q}}\int \limits_{2r}^{\infty}\int \limits_{B\left(
g,\tau \right)  }\left \vert b\left(  h\right)  -b_{B}\right \vert \left \vert
f\left(  h\right)  \right \vert dh\frac{d\tau}{\tau^{Q+1-\alpha}}%
\end{align*}
is valid. Applying the H\"{o}lder's inequality and by (\ref{5.1}),
(\ref{5.2}), we get%
\begin{align*}
J_{1}  & \lesssim r^{\frac{Q}{q}}\int \limits_{2r}^{\infty}\int
\limits_{B\left(  g,\tau \right)  }\left \vert b\left(  h\right)  -b_{B\left(
g,\tau \right)  }\right \vert \left \vert f\left(  h\right)  \right \vert
dh\frac{d\tau}{\tau^{Q+1-\alpha}}\\
& +r^{\frac{Q}{q}}\int \limits_{2r}^{\infty}\left \vert b_{B\left(  g,r\right)
}-b_{B\left(  g,\tau \right)  }\right \vert \int \limits_{B\left(  g,\tau \right)
}\left \vert f\left(  h\right)  \right \vert dh\frac{d\tau}{\tau^{Q+1-\alpha}}\\
& \lesssim r^{\frac{Q}{q}}\int \limits_{2r}^{\infty}\left \Vert \left(  b\left(
\cdot \right)  -b_{B\left(  g,\tau \right)  }\right)  \right \Vert _{L_{\mu
}\left(  B\left(  g,\tau \right)  \right)  }\left \Vert f\right \Vert
_{L_{p}\left(  B\left(  g,\tau \right)  \right)  }\frac{d\tau}{\tau
^{Q+1-\alpha}}\\
& +r^{\frac{Q}{q}}\int \limits_{2r}^{\infty}\left \vert b_{B\left(  g,r\right)
}-b_{B\left(  g,\tau \right)  }\right \vert \left \Vert f\right \Vert
_{L_{p}\left(  B\left(  g,\tau \right)  \right)  }\left \vert B\left(
g,\tau \right)  \right \vert ^{1-\frac{1}{p}}\frac{d\tau}{\tau^{Q+1-\alpha}}\\
& \lesssim \Vert b\Vert_{\ast}\,r^{\frac{Q}{q}}\int \limits_{2r}^{\infty}\left(
1+\ln \frac{\tau}{r}\right)  \left \Vert f\right \Vert _{L_{p}\left(  B\left(
g,\tau \right)  \right)  }\frac{d\tau}{\tau^{\frac{Q}{q}+1}}.
\end{align*}
In order to estimate $J_{2}$ note that%
\[
J_{2}=\left \Vert \left(  b\left(  \cdot \right)  -b_{B\left(  g,\tau \right)
}\right)  \right \Vert _{L_{q}\left(  B\left(  g,\tau \right)  \right)  }%
\int \limits_{\left(  2B\right)  ^{C}}\frac{\left \vert f\left(  h\right)
\right \vert }{\left \vert g^{-1}h\right \vert ^{Q-\alpha}}dh.
\]

By (\ref{5.1}), we get%
\[
J_{2}\lesssim \Vert b\Vert_{\ast}\,r^{\frac{Q}{q}}\int \limits_{\left(
2B\right)  ^{C}}\frac{\left \vert f\left(  h\right)  \right \vert }{\left \vert
g^{-1}h\right \vert ^{Q-\alpha}}dh.
\]
On the other hand, by the Fubini's theorem, we have%
\begin{align*}
\int \limits_{\left(  2B\right)  ^{C}}\frac{\left \vert f\left(  w\right)
\right \vert }{\left \vert g^{-1}w\right \vert ^{Q-\alpha}}dw  & \approx
\int \limits_{\left(  2B\right)  ^{C}}\left \vert f\left(  w\right)  \right \vert
\int \limits_{\left \vert g^{-1}w\right \vert }^{\infty}\frac{d\tau}%
{\tau^{Q+1-\alpha}}dw\\
& \approx \int \limits_{2r}^{\infty}\int \limits_{2r\leq \left \vert g^{-1}%
w\right \vert \leq \tau}\left \vert f\left(  w\right)  \right \vert dw\frac{d\tau
}{\tau^{Q+1-\alpha}}\\
& \lesssim \int \limits_{2r}^{\infty}\int \limits_{B\left(  g,\tau \right)
}\left \vert f\left(  w\right)  \right \vert dw\frac{d\tau}{\tau^{Q+1-\alpha}}.
\end{align*}
Applying the H\"{o}lder's inequality, we get%
\begin{align}
& \int \limits_{\left(  2B\right)  ^{C}}\frac{\left \vert f\left(  w\right)
\right \vert }{\left \vert g^{-1}w\right \vert ^{Q-\alpha}}dw\nonumber \\
& \lesssim \int \limits_{2r}^{\infty}\left \Vert f\right \Vert _{L_{p}\left(
B\left(  g,\tau \right)  \right)  }\frac{d\tau}{\tau^{\frac{Q}{q}+1}%
}.\label{e310}%
\end{align}

Thus, by (\ref{e310})%
\[
J_{2}\lesssim \Vert b\Vert_{\ast}\,r^{\frac{Q}{q}}\int \limits_{2r}^{\infty
}\left \Vert f\right \Vert _{L_{p}\left(  B\left(  g,\tau \right)  \right)
}\frac{d\tau}{\tau^{\frac{Q}{q}+1}}.
\]

Summing up $J_{1}$ and $J_{2}$, for all $p\in \left(  1,\infty \right)  $ we get%
\begin{equation}
\left \Vert T_{b,\alpha}f_{2}\right \Vert _{L_{q}\left(  B\right)  }%
\lesssim \Vert b\Vert_{\ast}\,r^{\frac{Q}{q}}\int \limits_{2r}^{\infty}\left(
1+\ln \frac{\tau}{r}\right)  \left \Vert f\right \Vert _{L_{p}\left(  B\left(
g,\tau \right)  \right)  }\frac{d\tau}{\tau^{\frac{Q}{q}+1}}.\label{6.3}%
\end{equation}
Finally, we have the following%
\[
\left \Vert T_{b,\alpha}f\right \Vert _{L_{q}\left(  B\right)  }\lesssim
\left \Vert b\right \Vert _{\ast}\left \Vert f\right \Vert _{L_{p}\left(
2B\right)  }+\Vert b\Vert_{\ast}\,r^{\frac{Q}{q}}\int \limits_{2r}^{\infty
}\left(  1+\ln \frac{\tau}{r}\right)  \left \Vert f\right \Vert _{L_{p}\left(
B\left(  g,\tau \right)  \right)  }\frac{d\tau}{\tau^{\frac{Q}{q}+1}}.
\]
On the other hand, we have%
\begin{align}
\left \Vert f\right \Vert _{L_{p}\left(  2B\right)  }  & \approx r^{\frac{Q}{q}%
}\left \Vert f\right \Vert _{L_{p}\left(  2B\right)  }\int \limits_{2r}^{\infty
}\frac{d\tau}{\tau^{\frac{Q}{q}+1}}\nonumber \\
& \leq r^{\frac{Q}{q}}\int \limits_{2r}^{\infty}\left \Vert f\right \Vert
_{L_{p}\left(  B\left(  g,\tau \right)  \right)  }\frac{d\tau}{\tau^{\frac
{Q}{q}+1}},\label{e313}%
\end{align}
which completes the proof of Lemma \ref{Lemma 5} by (\ref{e313}).
\end{proof}

Secondly, for the proof of Adams type results, we need some lemmas and
theorems about the estimates of sublinear commutator of fractional maximal
operator in generalized Morrey spaces on Heisenberg groups.

\begin{lemma}
\label{Lemma 1}Let $1<p<\infty$, $0\leq \alpha<\frac{Q}{p}$, $\frac{1}{q}%
=\frac{1}{p}-\frac{\alpha}{Q}$, $b\in BMO\left(  H_{n}\right)  $. Then the
inequality%
\[
\Vert M_{b,\alpha}f\Vert_{L_{q}(B(g,r))}\lesssim \Vert b\Vert_{\ast}%
\,r^{\frac{Q}{q}}\sup_{\tau>2r}\left(  1+\ln \frac{\tau}{r}\right)
\tau^{-\frac{Q}{q}}\Vert f\Vert_{L_{p}(B(g,\tau))}%
\]
holds for any ball $B(g,r)$ and for all $f\in L_{p}^{loc}(H_{n})$.
\end{lemma}

\begin{proof}
Let $1<p<\infty$, $0\leq \alpha<\frac{Q}{p}$ and $\frac{1}{q}=\frac{1}{p}%
-\frac{\alpha}{Q}$. For an arbitrary ball $B=B\left(  g,r\right)  $ we set
$f=f_{1}+f_{2}$, where \ $f_{1}=f\chi_{2B}$, \ $f_{2}=f\chi_{\left(
2B\right)  ^{C}}$ and $2B=B\left(  g,2r\right)  $. Hence,%
\[
\left \Vert M_{b,\alpha}f\right \Vert _{L_{q}\left(  B\right)  }\leq \left \Vert
M_{b,\alpha}f_{1}\right \Vert _{L_{q}\left(  B\right)  }+\left \Vert
M_{b,\alpha}f_{2}\right \Vert _{L_{q}\left(  B\right)  }.
\]
From the boundedness of $M_{b,\alpha}$ from $L_{p}(H_{n})$ to $L_{q}(H_{n})$
(see, for example, \cite{Alphonse, Folland-Stein, Stromberg}) it follows that:%
\begin{align*}
\left \Vert M_{b,\alpha}f_{1}\right \Vert _{L_{q}\left(  B\right)  }  &
\leq \left \Vert M_{b,\alpha}f_{1}\right \Vert _{L_{q}\left(  H_{n}\right)  }\\
& \lesssim \left \Vert b\right \Vert _{\ast}\left \Vert f_{1}\right \Vert
_{L_{p}\left(  H_{n}\right)  }=\left \Vert b\right \Vert _{\ast}\left \Vert
f\right \Vert _{L_{p}\left(  2B\right)  }.
\end{align*}

Let $h$ be an arbitrary point in $B$. If $B\left(  h,\tau \right)  \cap \left(
2B\right)  ^{C}\neq \emptyset$, then $\tau>r$. Indeed, if $w\in B\left(
h,\tau \right)  \cap \left(  2B\right)  ^{C}$, then $\tau>\left \vert
h^{-1}w\right \vert \geq \left \vert g^{-1}w\right \vert -\left \vert
g^{-1}h\right \vert >2r-r=r$. On the other hand, $B\left(  h,\tau \right)
\cap \left(  2B\right)  ^{C}\subset B\left(  g,2\tau \right)  $. Indeed, for
$w\in B\left(  h,\tau \right)  \cap \left(  2B\right)  ^{C}$ we have $\left \vert
g^{-1}w\right \vert \leq \left \vert h^{-1}w\right \vert +\left \vert
g^{-1}h\right \vert <\tau+r<2\tau$. Hence,%
\begin{align*}
M_{b,\alpha}f_{2}\left(  h\right)   & =\sup_{\tau>0}\frac{1}{\left \vert
B(h,\tau)\right \vert ^{1-\frac{\alpha}{Q}}}\int \limits_{B\left(
h,\tau \right)  \cap \left(  2B\right)  ^{C}}\left \vert b\left(  w\right)
-b\left(  h\right)  \right \vert \left \vert f\left(  w\right)  \right \vert dw\\
& \leq2^{Q-\alpha}\sup_{\tau>r}\frac{1}{\left \vert B(g,2\tau)\right \vert
^{1-\frac{\alpha}{Q}}}\int \limits_{B\left(  g,2\tau \right)  }\left \vert
b\left(  w\right)  -b\left(  h\right)  \right \vert \left \vert f\left(
w\right)  \right \vert dw\\
& =2^{Q-\alpha}\sup_{\tau>2r}\frac{1}{\left \vert B(g,\tau)\right \vert
^{1-\frac{\alpha}{Q}}}\int \limits_{B\left(  g,\tau \right)  }\left \vert
b\left(  w\right)  -b\left(  h\right)  \right \vert \left \vert f\left(
w\right)  \right \vert dw.
\end{align*}
Therefore, for all $h\in B$ we have%
\begin{equation}
M_{b,\alpha}f_{2}\left(  h\right)  \leq2^{Q-\alpha}\sup_{\tau>2r}\frac
{1}{\left \vert B(g,\tau)\right \vert ^{1-\frac{\alpha}{Q}}}\int
\limits_{B\left(  g,\tau \right)  }\left \vert b\left(  w\right)  -b\left(
h\right)  \right \vert \left \vert f\left(  w\right)  \right \vert dw.\label{63}%
\end{equation}

Then%
\begin{align*}
\left \Vert M_{b,\alpha}f_{2}\right \Vert _{L_{q}\left(  B\right)  }  &
\lesssim \left(  \int \limits_{B}\left(  \sup_{\tau>2r}\frac{1}{\left \vert
B(g,\tau)\right \vert ^{1-\frac{\alpha}{Q}}}\int \limits_{B\left(
g,\tau \right)  }\left \vert b\left(  w\right)  -b\left(  h\right)  \right \vert
\left \vert f\left(  w\right)  \right \vert dw\right)  ^{q}dg\right)  ^{\frac
{1}{q}}\\
& \leq \left(  \int \limits_{B}\left(  \sup_{\tau>2r}\frac{1}{\left \vert
B(g,\tau)\right \vert ^{1-\frac{\alpha}{Q}}}\int \limits_{B\left(
g,\tau \right)  }\left \vert b\left(  w\right)  -b_{B}\right \vert \left \vert
f\left(  w\right)  \right \vert dw\right)  ^{q}dg\right)  ^{\frac{1}{q}}\\
& +\left(  \int \limits_{B}\left(  \sup_{\tau>2r}\frac{1}{\left \vert
B(g,\tau)\right \vert ^{1-\frac{\alpha}{Q}}}\int \limits_{B\left(
g,\tau \right)  }\left \vert b\left(  h\right)  -b_{B}\right \vert \left \vert
f\left(  w\right)  \right \vert dw\right)  ^{q}dg\right)  ^{\frac{1}{q}}\\
& =J_{1}+J_{2}.
\end{align*}
Let us estimate $J_{1}$.%
\begin{align*}
J_{1}  & =r^{\frac{Q}{q}}\sup_{\tau>2r}\frac{1}{\left \vert B(g,\tau
)\right \vert ^{1-\frac{\alpha}{Q}}}\int \limits_{B\left(  g,\tau \right)
}\left \vert b\left(  w\right)  -b_{B}\right \vert \left \vert f\left(  w\right)
\right \vert dw\\
& \approx r^{\frac{Q}{q}}\sup_{\tau>2r}\tau^{\alpha-Q}\int \limits_{B\left(
g,\tau \right)  }\left \vert b\left(  w\right)  -b_{B}\right \vert \left \vert
f\left(  w\right)  \right \vert dw.
\end{align*}

Applying the H\"{o}lder's inequality, by (\ref{5.1}), (\ref{5.2}) and
$\frac{1}{\mu}+\frac{1}{p}=1$ we get%
\begin{align*}
J_{1}  & \lesssim r^{\frac{Q}{q}}\sup_{\tau>2r}\tau^{\alpha-Q}\int
\limits_{B\left(  g,\tau \right)  }\left \vert b\left(  w\right)  -b_{B\left(
g,\tau \right)  }\right \vert \left \vert f\left(  w\right)  \right \vert dw\\
& +r^{\frac{Q}{q}}\sup_{\tau>2r}\tau^{\alpha-Q}\left \vert b_{B\left(
g,r\right)  }-b_{B\left(  g,\tau \right)  }\right \vert \int \limits_{B\left(
g,\tau \right)  }\left \vert f\left(  w\right)  \right \vert dw\\
& \lesssim r^{\frac{Q}{q}}\sup_{\tau>2r}\tau^{\alpha-\frac{Q}{p}}\left \Vert
\left(  b\left(  \cdot \right)  -b_{B\left(  g,\tau \right)  }\right)
\right \Vert _{L_{\mu}\left(  B\left(  g,\tau \right)  \right)  }\left \Vert
f\right \Vert _{L_{p}\left(  B\left(  g,\tau \right)  \right)  }\\
& +r^{\frac{Q}{q}}\sup_{\tau>2r}t^{\alpha-Q}\left \vert b_{B\left(  g,r\right)
}-b_{B\left(  g,\tau \right)  }\right \vert \left \Vert f\right \Vert
_{L_{p}\left(  B\left(  g,\tau \right)  \right)  }\left \vert B\left(
g,\tau \right)  \right \vert ^{1-\frac{1}{p}}\\
& \lesssim \Vert b\Vert_{\ast}\,r^{\frac{Q}{q}}\sup_{\tau>2r}\left(  1+\ln
\frac{\tau}{r}\right)  t^{-\frac{Q}{q}}\Vert f\Vert_{L_{p}(B(g,\tau))}.
\end{align*}
In order to estimate $J_{2}$ note that%
\[
J_{2}=\left \Vert \left(  b\left(  \cdot \right)  -b_{B\left(  g,\tau \right)
}\right)  \right \Vert _{L_{q}\left(  B\left(  g,\tau \right)  \right)  }%
\sup_{\tau>2r}t^{\alpha-Q}\int \limits_{B\left(  g,\tau \right)  }\left \vert
f\left(  w\right)  \right \vert dw.
\]

By (\ref{5.1}), we get%
\[
J_{2}\lesssim \Vert b\Vert_{\ast}\,r^{\frac{Q}{q}}\sup_{\tau>2r}t^{\alpha
-Q}\int \limits_{B\left(  g,\tau \right)  }\left \vert f\left(  w\right)
\right \vert dw.
\]

Thus, by (\ref{e310})%
\[
J_{2}\lesssim \Vert b\Vert_{\ast}\,r^{\frac{Q}{q}}\sup_{\tau>2r}t^{-\frac{Q}%
{q}}\left \Vert f\right \Vert _{L_{p}\left(  B\left(  g,\tau \right)  \right)  }.
\]

Summing up $J_{1}$ and $J_{2}$, for all $p\in \left(  1,\infty \right)  $ we get%
\begin{equation}
\left \Vert M_{b,\alpha}f_{2}\right \Vert _{L_{q}\left(  B\right)  }%
\lesssim \Vert b\Vert_{\ast}\,r^{\frac{Q}{q}}\sup_{\tau>2r}t^{-\frac{Q}{q}%
}\left(  1+\ln \frac{\tau}{r}\right)  \left \Vert f\right \Vert _{L_{p}\left(
B\left(  g,\tau \right)  \right)  }.\label{64}%
\end{equation}

Finally, we have the following%
\begin{align*}
\left \Vert M_{b,\alpha}f\right \Vert _{L_{q}\left(  B\right)  }  &
\lesssim \left \Vert b\right \Vert _{\ast}\left \Vert f\right \Vert _{L_{p}\left(
2B\right)  }+\Vert b\Vert_{\ast}\,r^{\frac{Q}{q}}\sup_{\tau>2r}t^{-\frac{Q}%
{q}}\left(  1+\ln \frac{\tau}{r}\right)  \left \Vert f\right \Vert _{L_{p}\left(
B\left(  g,\tau \right)  \right)  }\\
& \lesssim \Vert b\Vert_{\ast}\,r^{\frac{Q}{q}}\sup_{\tau>2r}t^{-\frac{Q}{q}%
}\left(  1+\ln \frac{\tau}{r}\right)  \left \Vert f\right \Vert _{L_{p}\left(
B\left(  g,\tau \right)  \right)  },
\end{align*}
which completes the proof.
\end{proof}

Similarly to Lemma \ref{Lemma 1} the following lemma can also be proved.

\begin{lemma}
Let $1<p<\infty$, $b\in BMO\left(  H_{n}\right)  $ and $M_{b}$ is bounded on
$L_{p}(H_{n})$. Then the inequality%
\[
\Vert M_{b}f\Vert_{L_{p}(B(g,r))}\lesssim \Vert b\Vert_{\ast}\,r^{\frac{Q}{q}%
}\sup_{\tau>2r}\left(  1+\ln \frac{\tau}{r}\right)  \tau^{-\frac{Q}{p}}\Vert
f\Vert_{L_{p}(B(g,\tau))}%
\]
holds for any ball $B(g,r)$ and for all $f\in L_{p}^{loc}(H_{n})$.
\end{lemma}

The following theorem is true.

\begin{theorem}
\label{teo10}Let $1<p<\infty$, $0\leq \alpha<\frac{Q}{p}$, $\frac{1}{q}%
=\frac{1}{p}-\frac{\alpha}{Q}$, $b\in BMO\left(  H_{n}\right)  $ and let
$\left(  \varphi_{1},\varphi_{2}\right)  $ satisfies the condition%
\[
\sup_{r<t<\infty}t^{\alpha-\frac{Q}{p}}\left(  1+\ln \frac{t}{r}\right)
\operatorname*{essinf}\limits_{t<\tau<\infty}\varphi_{1}\left(  g,\tau \right)
\tau^{\frac{Q}{p}}\leq C\varphi_{2}\left(  g,r\right)  ,
\]
where $C$ does not depend on $g$ and $r$. Then the operator $M_{b,\alpha}$ is
bounded from $M_{p,\varphi_{1}}\left(  H_{n}\right)  $ to $M_{q,\varphi_{2}%
}\left(  H_{n}\right)  $. Moreover%
\[
\left \Vert M_{b,\alpha}f\right \Vert _{M_{q,\varphi_{2}}}\lesssim \Vert
b\Vert_{\ast}\left \Vert f\right \Vert _{M_{p,\varphi_{1}}}.
\]

\end{theorem}

\begin{proof}
By Theorem 4.1 in \cite{Guliyev} and Lemma \ref{Lemma 1}, we get%
\begin{align*}
\left \Vert M_{b,\alpha}f\right \Vert _{M_{q,\varphi_{2}}}  & \lesssim \Vert
b\Vert_{\ast}\sup_{g\in H_{n},r>0}\varphi_{2}\left(  g,r\right)  ^{-1}%
\sup_{\tau>r}\left(  1+\ln \frac{\tau}{r}\right)  \tau^{-\frac{Q}{q}}\Vert
f\Vert_{L_{p}(B(g,\tau))}\\
& \lesssim \Vert b\Vert_{\ast}\sup_{g\in H_{n},r>0}\varphi_{1}\left(
g,r\right)  ^{-1}r^{-\frac{Q}{p}}\Vert f\Vert_{L_{p}(B(g,r))}=\Vert
b\Vert_{\ast}\left \Vert f\right \Vert _{M_{p,\varphi_{1}}}.
\end{align*}

\end{proof}

In the case of $\alpha=0$ and $p=q$, we get the following corollary by Theorem
\ref{teo10}.

\begin{corollary}
\label{corollary2}Let $1<p<\infty$, $b\in BMO\left(  H_{n}\right)  $ and
$\left(  \varphi_{1},\varphi_{2}\right)  $ satisfies the condition%
\[
\sup_{r<t<\infty}t^{-\frac{Q}{p}}\left(  1+\ln \frac{t}{r}\right)
\operatorname*{essinf}\limits_{t<\tau<\infty}\varphi_{1}\left(  g,\tau \right)
\tau^{\frac{Q}{p}}\leq C\varphi_{2}\left(  g,r\right)  ,
\]
where $C$ does not depend on $g$ and $r$. Then the operator $M_{b}$ is bounded
from $M_{p,\varphi_{1}}\left(  H_{n}\right)  $ to $M_{p,\varphi_{2}}\left(
H_{n}\right)  $. Moreover%
\[
\left \Vert M_{b}f\right \Vert _{M_{p,\varphi_{2}}}\lesssim \Vert b\Vert_{\ast
}\left \Vert f\right \Vert _{M_{p,\varphi_{1}}}.
\]

\end{corollary}

\section{Proofs of the main results}

\subsection{\textbf{Proof of Theorem \ref{teo15}.}}

\begin{proof}
To prove Theorem \ref{teo15}, we will use the following relationship between
essential supremum and essential infimum%
\begin{equation}
\left(  \operatorname*{essinf}\limits_{x\in E}f\left(  x\right)  \right)
^{-1}=\operatorname*{esssup}\limits_{x\in E}\frac{1}{f\left(  x\right)
},\label{5}%
\end{equation}
where $f$ is any real-valued nonnegative function and measurable on $E$ (see
\cite{Wheeden-Zygmund}, page 143). Indeed, since $f\in M_{p,\varphi_{1}}$, by
(\ref{5}) and the non-decreasing, with respect to $\tau$, of the norm
$\left \Vert f\right \Vert _{L_{p}\left(  B\left(  g,\tau \right)  \right)  }$,
we get%
\begin{align}
& \frac{\left \Vert f\right \Vert _{L_{p}\left(  B\left(  g,\tau \right)
\right)  }}{\operatorname*{essinf}\limits_{0<\tau<s<\infty}\varphi
_{1}(g,s)s^{\frac{Q}{p}}}\leq \operatorname*{esssup}\limits_{0<\tau<s<\infty
}\frac{\left \Vert f\right \Vert _{L_{p}\left(  B\left(  g,\tau \right)  \right)
}}{\varphi_{1}(g,s)s^{\frac{Q}{p}}}\nonumber \\
& \leq \operatorname*{esssup}\limits_{0<s<\infty}\frac{\left \Vert f\right \Vert
_{L_{p}\left(  B\left(  g,s\right)  \right)  }}{\varphi_{1}(g,s)s^{\frac{Q}%
{p}}}\leq \left \Vert f\right \Vert _{M_{p,\varphi_{1}}}.\label{10}%
\end{align}
For $1<p<\infty$, since $(\varphi_{1},\varphi_{2})$ satisfies (\ref{47}) and
by (\ref{10}), we have%
\begin{align}
& \int \limits_{r}^{\infty}\left(  1+\ln \frac{\tau}{r}\right)  \left \Vert
f\right \Vert _{L_{p}\left(  B\left(  g,\tau \right)  \right)  }\tau^{-\frac
{Q}{q}}\frac{d\tau}{\tau}\nonumber \\
& \leq \int \limits_{r}^{\infty}\left(  1+\ln \frac{\tau}{r}\right)
\frac{\left \Vert f\right \Vert _{L_{p}\left(  B\left(  g,\tau \right)  \right)
}}{\operatorname*{essinf}\limits_{\tau<s<\infty}\varphi_{1}(g,s)s^{\frac{Q}%
{p}}}\frac{\operatorname*{essinf}\limits_{\tau<s<\infty}\varphi_{1}%
(g,s)s^{\frac{Q}{p}}}{\tau^{\frac{Q}{q}}}\frac{d\tau}{\tau}\nonumber \\
& \leq C\left \Vert f\right \Vert _{M_{p,\varphi_{1}}}\int \limits_{r}^{\infty
}\left(  1+\ln \frac{\tau}{r}\right)  \frac{\operatorname*{essinf}%
\limits_{\tau<s<\infty}\varphi_{1}(g,s)s^{\frac{Q}{p}}}{\tau^{\frac{Q}{q}}%
}\frac{d\tau}{\tau}\nonumber \\
& \leq C\left \Vert f\right \Vert _{M_{p,\varphi_{1}}}\varphi_{2}%
(g,r).\label{41}%
\end{align}
Then by (\ref{40}) and (\ref{41}), we get%
\begin{align*}
\left \Vert T_{b,\alpha}f\right \Vert _{M_{q,\varphi_{2}}}  & =\sup_{g\in
H_{n}{,}r>0}\varphi_{2}\left(  g,r\right)  ^{-1}|B(g,r)|^{-\frac{1}{q}%
}\left \Vert T_{b,\alpha}f\right \Vert _{L_{q}\left(  B\left(  g,r\right)
\right)  }\\
& \lesssim \Vert b\Vert_{\ast}\sup_{g\in H_{n}{,}r>0}\varphi_{2}\left(
g,r\right)  ^{-1}\int \limits_{r}^{\infty}\left(  1+\ln \frac{\tau}{r}\right)
\left \Vert f\right \Vert _{L_{p}\left(  B\left(  g,\tau \right)  \right)  }%
\tau^{-\frac{Q}{q}}\frac{d\tau}{\tau}\\
& \lesssim \Vert b\Vert_{\ast}\left \Vert f\right \Vert _{M_{p,\varphi_{1}}}.
\end{align*}
This completes the proof of Theorem \ref{teo15}.
\end{proof}

\subsection{\textbf{Proof of Theorem \ref{teo100}.}}

\begin{proof}
Let $1<p<\infty$, $0<\alpha<\frac{Q}{p}$ and $\frac{1}{q}=\frac{1}{p}%
-\frac{\alpha}{Q}$, $p<q$ and $f\in M_{p,\varphi^{\frac{1}{p}}}$. For an
arbitrary ball $B=B\left(  g,r\right)  $ we set $f=f_{1}+f_{2}$, where
\ $f_{1}=f\chi_{2B}$, \ $f_{2}=f\chi_{\left(  2B\right)  ^{C}}$ and
$2B=B\left(  g,2r\right)  $. Then we have
\[
\left \Vert T_{b,\alpha}f\right \Vert _{L_{q}\left(  B\right)  }\leq \left \Vert
T_{b,\alpha}f_{1}\right \Vert _{L_{q}\left(  B\right)  }+\left \Vert
T_{b,\alpha}f_{2}\right \Vert _{L_{q}\left(  B\right)  }.
\]

For $g\in B$ we have%
\[
\left \vert T_{b,\alpha}f_{2}\left(  g\right)  \right \vert \lesssim
\int \limits_{\left(  2B\right)  ^{C}}\left \vert b\left(  h\right)  -b\left(
g\right)  \right \vert \frac{\left \vert f\left(  h\right)  \right \vert
}{\left \vert g^{-1}h\right \vert ^{Q-\alpha}}dh.
\]

Analogously to Section 2, for all $p\in \left(  1,\infty \right)  $ and $g\in B
$ we get%
\begin{equation}
\left \vert T_{b,\alpha}f_{2}\left(  x\right)  \right \vert \lesssim \Vert
b\Vert_{\ast}\, \int \limits_{r}^{\infty}\left(  1+\ln \frac{\tau}{r}\right)
\tau^{\alpha-\frac{Q}{p}-1}\left \Vert f\right \Vert _{L_{p}\left(  B\left(
g,\tau \right)  \right)  }d\tau.\label{75}%
\end{equation}

Then from conditions (\ref{74}), (\ref{100}) and inequality (\ref{75}) we get%
\begin{align}
\left \vert T_{b,\alpha}f\left(  g\right)  \right \vert  & \lesssim \Vert
b\Vert_{\ast}r^{\alpha}M_{b}f\left(  g\right)  +\Vert b\Vert_{\ast}%
\int \limits_{r}^{\infty}\left(  1+\ln \frac{\tau}{r}\right)  \tau^{\alpha
-\frac{Q}{p}-1}\left \Vert f\right \Vert _{L_{p}\left(  B\left(  g,\tau \right)
\right)  }d\tau \nonumber \\
& \leq \Vert b\Vert_{\ast}r^{\alpha}M_{b}f\left(  g\right)  +\Vert b\Vert
_{\ast}\left \Vert f\right \Vert _{M_{p,\varphi^{\frac{1}{p}}}}\int
\limits_{r}^{\infty}\left(  1+\ln \frac{\tau}{r}\right)  \tau^{\alpha}%
\varphi \left(  g,\tau \right)  ^{\frac{1}{p}}\frac{d\tau}{\tau}\nonumber \\
& \lesssim \Vert b\Vert_{\ast}r^{\alpha}M_{b}f\left(  g\right)  +\Vert
b\Vert_{\ast}r^{-\frac{\alpha p}{q-p}}\left \Vert f\right \Vert _{M_{p,\varphi
^{\frac{1}{p}}}}.\label{76}%
\end{align}

Hence choosing $r=\left(  \frac{\left \Vert f\right \Vert _{M_{p,\varphi
^{\frac{1}{p}}}}}{M_{b}f\left(  g\right)  }\right)  ^{\frac{q-p}{\alpha q}}$
for every $g\in H_{n}$, we have%
\[
\left \vert T_{b,\alpha}f\left(  g\right)  \right \vert \lesssim \Vert
b\Vert_{\ast}\left(  M_{b}f\left(  g\right)  \right)  ^{\frac{p}{q}}\left \Vert
f\right \Vert _{M_{p,\varphi^{\frac{1}{p}}}}^{1-\frac{p}{q}}.
\]

Consequently the statement of the theorem follows in view of the boundedness
of the commutator of the maximal operator $M_{b}$ in $M_{p,\varphi^{\frac
{1}{p}}}\left(  H_{n}\right)  $ provided by Corollary \ref{corollary2} in
virtue of condition (\ref{67}).

Therefore, we have%
\begin{align*}
\left \Vert T_{b,\alpha}f\right \Vert _{M_{q,\varphi^{\frac{1}{q}}}}  &
=\sup_{g\in H_{n},\tau>0}\varphi \left(  g,\tau \right)  ^{-\frac{1}{q}}%
\tau^{-\frac{Q}{q}}\left \Vert T_{b,\alpha}f\right \Vert _{L_{q}\left(  B\left(
g,\tau \right)  \right)  }\\
& \lesssim \Vert b\Vert_{\ast}\left \Vert f\right \Vert _{M_{p,\varphi^{\frac
{1}{p}}}}^{1-\frac{p}{q}}\sup_{g\in H_{n},\tau>0}\varphi \left(  g,\tau \right)
^{-\frac{1}{q}}\tau^{-\frac{Q}{q}}\left \Vert M_{b}f\right \Vert _{L_{p}\left(
B\left(  g,\tau \right)  \right)  }^{\frac{p}{q}}\\
& =\Vert b\Vert_{\ast}\left \Vert f\right \Vert _{M_{p,\varphi^{\frac{1}{p}}}%
}^{1-\frac{p}{q}}\left(  \sup_{g\in H_{n},\tau>0}\varphi \left(  g,\tau \right)
^{-\frac{1}{p}}\tau^{-\frac{Q}{p}}\left \Vert M_{b}f\right \Vert _{L_{p}\left(
B\left(  g,\tau \right)  \right)  }\right)  ^{\frac{p}{q}}\\
& =\Vert b\Vert_{\ast}\left \Vert f\right \Vert _{M_{p,\varphi^{\frac{1}{p}}}%
}^{1-\frac{p}{q}}\left \Vert M_{b}f\right \Vert _{M_{p,\varphi^{\frac{1}{p}}}%
}^{\frac{p}{q}}\\
& \lesssim \Vert b\Vert_{\ast}\left \Vert f\right \Vert _{M_{p,\varphi^{\frac
{1}{p}}}}.
\end{align*}

\end{proof}

\begin{remark}
In the case of $\varphi \left(  g,r\right)  =r^{\lambda-Q}$, $0<\lambda<Q$ from
Theorem \ref{teo100} we get the following Adams type result (\cite{Adams}) for
the commutators of fractional maximal and integral operators.
\end{remark}

\begin{corollary}
Let $0<\alpha<Q$, $1<p<\frac{Q}{\alpha}$, $0<\lambda<Q-\alpha p$, $\frac{1}%
{p}-\frac{1}{q}=\frac{\alpha}{Q-\lambda}$ and $b\in BMO\left(  H_{n}\right)
$. Then, the operators $M_{b,\alpha}$ and $[b,\overline{T}_{\alpha}]$ are
bounded from $L_{p,\lambda}\left(  H_{n}\right)  $ to $L_{q,\lambda}\left(
H_{n}\right)  $.
\end{corollary}


\begin{thebibliography}{99}                                                                                               %
\bibitem {Adams}D.R. Adams, A note on Riesz potentials, Duke Math. J., 42
(1975), 765-778.

\bibitem {Alphonse}A.M. Alphonse, An end point estimate for maximal
commutators, J. Fourier Anal. Appl., 6 (4) (2000), 449-456.

\bibitem {ChFra}F. Chiarenza, M. Frasca,\ Morrey spaces and Hardy-Littlewood
maximal function, Rend. Mat., 7 (1987), 273-279.

\bibitem {ChFraL1}F. Chiarenza, M. Frasca, P. Longo,\ Interior $W^{2,p}%
$-estimates for nondivergence elliptic equations with discontinuous
coefficients, Ricerche Mat., 40 (1991), 149-168.

\bibitem {ChFraL2}F. Chiarenza, M. Frasca, P. Longo,\ $W^{2,p}$-solvability of
Dirichlet problem for nondivergence elliptic equations with VMO coefficients,
Trans. Amer. Math. Soc.,\ 336 (1993), 841-853.

\bibitem {Muller1}T. Coulhon, D. M\"{u}ller and J. Zienkiewicz, About Riesz
transforms on the Heisenberg groups, Math. Ann., 305 (1) (1996), 369-379.

\bibitem {Folland}G.B. Folland, Subelliptic estimates and function spaces on
nilpotent Lie groups, Ark. Mat., 13 (1975), 161-207.

\bibitem {Folland-Stein}G.B. Folland, E.M. Stein, Hardy Spaces on homogeneous
groups, Mathematical Notes, 28, Princeton Univ. Press, Princeton, 1982.

\bibitem {Garofalo}N. Garofalo and E. Lanconelli, Frequency functions on the
Heisenberg group, the uncertainty principle and unique continuation. Atin.
Inst. Fourier, Grenoble, 40 (1990), 313-356.

\bibitem {Giaquinta}M. Giaquinta, Multiple integrals in the calculus of
variations and nonlinear elliptic systems, Princeton Univ. Press, Princeton,
NJ, 1983.

\bibitem {Guliyev}V.S. Guliyev, A. Eroglu, Y.Y. Mammadov, Riesz potential in
generalized Morrey spaces on the Heisenberg group, J. Math. Sci., 189 (3)
(2013), 365-382.

\bibitem {Gurbuz0}F. Gurbuz, The boundedness of Hardy-Littlewood maximal
operator and Riesz potential in Morrey spaces $\left[  \text{\textit{MD.
thesis}}\right]  $, Ankara University, Ankara, Turkey, 2011 (in Turkish).

\bibitem {Gurbuz}F. Gurbuz\textit{, }Weighted Morrey and Weighted fractional
Sobolev-Morrey Spaces estimates for a large class of pseudo-differential
operators with smooth symbols, J. Pseudo-Differ. Oper. Appl., 7 (4) (2016),
595-607. doi:10.1007/s11868-016-0158-8.

\bibitem {John-Nirenberg}F. John and L. Nirenberg, On functions of bounded
mean oscillation. Comm. Pure Appl. Math., 14 (1961), 415-426.

\bibitem {Koranyi}A. Koranyi, H.M. Reimann, Quasiconformal mappings on the
Heisenberg group, Invent. Math., 80 (1985), 309-338.

\bibitem {Kufner}A. Kufner, O. John and S. Fucik, Function Spaces, Noordhoff,
Leyden, and Academia, Prague, 1977.

\bibitem {LLY}G. Lu, S.Z. Lu, D.C. Yang, Singular integrals and commutators on
homogeneous groups, Anal. Math., 28 (2002), 103-134.

\bibitem {Morrey}C.B. Morrey, On the solutions of quasi-linear elliptic
partial differential equations, Trans. Amer. Math. Soc., 43 (1938), 126-166.

\bibitem {Muller2}D. M\"{u}ller, F. Ricci and E.M. Stein, Marcinkiewicz
multipliers and multi-parameter structure on Heisenberg (-type) groups. I,
Inventiones Mathematicae, 119 (1) (1995), 199-233.

\bibitem {Muller3}D. M\"{u}ller, F. Ricci and E.M. Stein, Marcinkiewicz
multipliers and multi-parameter structure on Heisenberg (-type) groups. II,
Mathematische Zeitschrift, 221 (1) (1996), 267-291.

\bibitem {Peetre}J. Peetre\textit{,\ }On the theory of $M_{p,\lambda}$, J.
Funct. Anal., 4 (1969), 71-87.

\bibitem {SW}F. Soria, G. Weiss,\textit{\ }A remark on singular integrals and
power weights, Indiana Univ. Math. J., 43 (1994) 187-204.

\bibitem {Stein93}E.M. Stein, Harmonic Analysis: Real Variable Methods,
Orthogonality and Oscillatory Integrals, Princeton Univ. Press, Princeton NJ, 1993.

\bibitem {Stromberg}J.O. Str\"{o}mberg, A. Torchinsky, Weighted Hardy Spaces.
Lecture Notes in Math, Vol 1381. Berlin: Springer-Verlag, 1989.

\bibitem {Wheeden-Zygmund}R.L. Wheeden and A. Zygmund, Measure and Integral:
An Introduction to Real Analysis, vol. 43 of Pure and Applied Mathematics,
Marcel Dekker, New York, NY, USA, 1977.

\bibitem {Xiao}J. Xiao and J. He, Riesz potential on the Heisenberg group, J.
Inequal. Appl. 2011, ID 498638 (2011).

\bibitem {Yener}A. Yener, Weighted Hardy type inequalities on the Heisenberg
group $H^{n}$, Math. Inequal. Appl., 19 (2) (2016), 671-683.
\end{thebibliography}
\end{document}